\documentclass[12pt, leqno]{amsart}

\usepackage{amsfonts,amsmath,amsthm,amssymb,bbm}
\usepackage{latexsym}
\usepackage{enumerate}
\usepackage[usenames]{color}
\usepackage{hyperref,latexsym,multicol}
\usepackage[usestackEOL]{stackengine}
\usepackage[pagewise]{lineno}

\usepackage{xcolor}

\newtheorem{case}{Case}
\newtheorem{theorem}{Theorem}[section]

\newtheorem{lemma}[theorem]{Lemma}

\usepackage[ansinew]{inputenc}
\usepackage{ifthen}
\usepackage{animate}
\usepackage{biblatex}
\newcommand{\be}{\begin{equation}}
	\newcommand{\ee}{\end{equation}}
\newcommand{\Mod}[1]{\ (\mathrm{mod}\ #1)}

\title[Products of Tribonacci Numbers]{Products of Tribonacci Numbers that are the Products of Factorials}

\author[Das]{Pranabesh Das}
\address{Department of Mathematics, Xavier University of Louisiana,  1 Drexel Dr, New Orleans, LA 70125, USA}
\email {pranabesh.math@gmail.com}
\author[Saunders]{J.C. Saunders}
\address{Department of Mathematical Sciences, Middle Tennessee State University, Murfreesboro, TN, USA}
\email{John.Saunders@mtsu.edu}

\usepackage{amssymb}
\usepackage{amsmath}
\usepackage{amsthm}
\usepackage{mathtools}
\theoremstyle{definition}
\newtheorem{definition}[theorem]{Definition}

\begin{document}

\begin{abstract}
		In 2014 Marques and Lengyel gave all of the solutions to the equation $T_n=m!$, where $T_n$ is the $n$th term of the Tribonacci sequence $0,1,1,2,4,7,13,24,\ldots$. In 2023 Alahmadi and Luca generalized their result to the equation $T_n=m_1!m_2!\cdots m_k!$ for every $k\in\mathbb{N}$, where $m_1\leq m_2\leq\ldots\leq m_k$ listing all the solutions to this equation. Here we generalize these results further and give all the solutions to $T_nT_{n+1}T_{n+2}\cdots T_{n+r}=m_1!m_2!\cdots m_k!$ and $ |T_{-n}T_{-n-1}T_{-n-2}\cdots T_{-n-r}|=m_1!m_2!\cdots m_k!$ for every $n,r\in\mathbb{N}$, where $m_1\leq m_2\leq\ldots\leq m_k$.
	\end{abstract}
    
\maketitle

\section{Introduction}
The Fibonacci numbers $\{F_n\}_{n\geq 0}$ are well known and well studied in the literature and are defined by the relation
	\begin{equation*}
		F_n= F_{n-1}+F_{n-2} \quad {\rm for} \ n\geq 2
	\end{equation*}
	with the initial condition $F_0=0$ and $F_1=1$.
	
	Many mathematicians studied Diophantine equations involving
	factorial and Fibonacci numbers, or to be more general factorials and Lucas sequences. In 1999, Luca \cite{FL99} proved that $F_n$ is a product of factorials only when $n = 1, 2, 3, 6, 12$. Luca and St\u{a}nic\u{a} \cite{LS05} considered the equation 
	\begin{equation}\label{LS-1}
		F_{n_1}\cdots F_{n_k}=m_1!\cdots m_t!
	\end{equation}
	for positive integers $n_1, \cdots, n_k$ and $m_1,\cdots,m_t$. Using the Primitive Divisor Theorem (given, for instance, in \cite{B01}), Luca and St\u{a}nic\u{a} \cite{LS05} showed that the largest solution of \eqref{LS-1} is given by $F_1F_2F_3F_4F_5F_6F_8F_{10}F_{12}=11!$. Laishram, Luca, and Sias \cite{LLS} studied the equation $|u_n|=m_1!\cdots m_t!$ where the $u_n$'s are terms of a Lucas sequence.
	
	The study of intersections between number sequences and products (or sums) of factorials have been a common theme over the past two decades and many other classes binary recurrence sequence have been studied. Generally the techniques involved ranged from Linear Forms in Logarithms, Baker's method, and the Primitive Divisor Theorem. Those techniques do not translate very well for higher order recurrence sequences.
	
	The Tribonacci numbers ($\{T_n\}_{n\geq 0}$ are defined by the recursive relation
		\begin{equation}\label{tribo}
			T_{n+3} = T_{n+2} + T_{n+1} + T_n, \qquad n\geq 0
		\end{equation} 	
		with initial conditions $T_0 = 0, T_1 = T_2 = 1$. The first ten Tribonacci numbers are
        \begin{equation*}
        0, 1, 1, 2, 4, 7, 13, 24, 44, 81, 149.
        \end{equation*}
    In 2014, Marques and Lengyel \cite{ML} showed that the equation
		\begin{equation}\label{ML-Tribo}
			T_n=m!, \quad m,n \in\mathbb{N}
		\end{equation}
		has only solutions for $n\in\{1, 2, 3, 4, 7\}$. Alahmadi and Luca \cite{AL-23} considered the general equation
		\begin{equation}\label{AL-Pos-23}
			T_n=m_1!m_2!\cdots m_k!, \quad n \in\mathbb{N}
		\end{equation}
		for some positive integers $m_1\leq m_2 \leq \cdots \leq m_k$. They extended the previous results for Marques and Lengyel \cite{ML}, and for \eqref{AL-Pos-23}, they showed that $n\in\{1, 2, 3, 4, 7\}$.
		
		Tribonacci sequence has been generalized in the negative direction using the recurrence $T_{-n} = -T_{-(n-1)} - T_{-(n-2)} + T_{-(n-3)}$ for $n\geq3$ with initial terms $T_{-1}=0$ and $T_{-2}=1$. This sequence changes sign infinitely often and the first few terms are 
        \begin{equation*}
        \cdots,-103, 0, 56,-47, 9, 18, -20, 7, 5,-8, 4, 1,-3, 2, 0,-1, 1, 0.
        \end{equation*}
		
		Luca and Alahmadi \cite{AL-23} also proved that the equation
		\begin{equation}\label{AL-Neg-23}
			|T_n|=m_1!m_2!\cdots m_k!, \quad n \in\mathbb{Z}
		\end{equation}
		only has solutions for $n\in\{-9,-8,-7,-5,-3,-2, 1, 2, 3, 4, 7\}$.
		
		In this article, we study the generalization of equations \eqref{AL-Pos-23}, \eqref{AL-Neg-23} in products of terms in Tribonacci numbers. More precisely, we consider the following two equations:
		
		\begin{equation}\label{Tribo-1}
			T_nT_{n+1}\cdots T_{n+r}=m_1!m_2!\cdots m_k!, \quad n,r,m_1,m_2\ldots,m_k\in\mathbb{N}
		\end{equation}
		
		\begin{equation}\label{Tribo-2}
			|T_{-n}T_{-n-1}\cdots T_{-n-r}|=m_1!m_2!\cdots m_k!, \quad n,r,m_1,m_2\ldots,m_k\in\mathbb{N}.
		\end{equation}
		\section{\bf Main Results}
		In this section we state the results we obtain in this article.
		\begin{theorem}\label{Thm-1}
			Let $n,r\in\mathbb{N}$ with $n\geq 3$ and 
			\begin{equation*}
				T_nT_{n+1}T_{n+2}\cdots T_{n+r}=m_1!m_2!\cdots m_k!,
			\end{equation*}
			where $m_1\leq m_2\leq\ldots\leq m_k$. The only solution is
			\begin{equation*}
				T_3T_4=2!\cdot 2!\cdot 2!.
			\end{equation*}
		\end{theorem}
        Note that the cases $n=0,1,2$ in Theorem \ref{Thm-1} are omitted to avoid trivial cases.		
		\begin{theorem}\label{Thm-2}
			Let $n,r\in\mathbb{N}$ with $n\geq 5$ and 
			\begin{equation*}
				|T_{-n}T_{-n-1}T_{-n-2}\cdots T_{-n-r}|=m_1!m_2!\cdots m_k!,
			\end{equation*}
			where $m_1\leq m_2\leq\ldots\leq m_k$. The only solutions to this equation are
			\begin{align*}
				&\left|T_{-5}T_{-6}\right|=\left|T_{-5}T_{-6}T_{-7}\right|=3!\\
				&\left|T_{-5}\cdots T_{-8}\right|=4!\\
				&\left|T_{-5}\cdots T_{-9}\right|=(2!)^3\cdot 4!\\
				&\left|T_{-5}\cdots T_{-10}\right|=(2!)^3\cdot 5!\\
				&\left|T_{-5}\cdots T_{-13}\right|=(2!)^2\cdot 5!\cdot 7!\\
				&\left|T_{-5}\cdots T_{-14}\right|=3!\cdot 6!\cdot 7!\\
				&\left|T_{-6}T_{-7}T_{-8}\right|=2!\cdot 3!\\
				&\left|T_{-6}\cdots T_{-9}\right|=(2!)^2\cdot 4!\\
				&\left|T_{-6}\cdots T_{-10}\right|=(2!)^2\cdot 5!\\
				&\left|T_{-6}\cdots T_{-13}\right|=2!\cdot \cdot 5!\cdot 7!\\
				&\left|T_{-7}T_{-8}\right|=(2!)^2\\
				&\left|T_{-7}T_{-8}T_{-9}\right|=\left|T_{-8}T_{-9}\right|=(2!)^5\\
				&\left|T_{-8}T_{-9}\right|=(2!)^5\\
				&\left|T_{-7}\cdots T_{-14}\right|=\left|T_{-8}\cdots T_{-14}\right|=6!\cdot 7!.
			\end{align*}
		\end{theorem}
		
The following theorem is a generalization of Theorem \ref{Thm-1}. We obtain a finiteness result for the general case but it may be possible to explicitly find the solutions.
		\begin{theorem}\label{Thm-3}
			Let $n,r,d\in\mathbb{N}$ with $n\geq 3$ and 
			\begin{equation*}
				T_nT_{n+d}T_{n+2d}\cdots T_{n+rd}=m_1!m_2!\cdots m_k!,
			\end{equation*}
			where $2\leq m_1\leq m_2\leq\ldots\leq m_k$. The equation has only finitely many solutions, in particular, we have $rd\leq 1439862$ and $n<719933$.
		\end{theorem}

		The proofs of Theorems \ref{Thm-1} and \ref{Thm-2} depend on a careful $p$-adic analysis on the $2$-adic valuations of Tribonacci numbers. We have the following definitions, the second one of which will be used in later on in Section 4.
        \begin{definition}
        For a prime $p$ and a number $n$, we write $\nu_p(n)=k$ if $p^k$ is the highest power of $p$ dividing $n$ and say $k$ is the $p$-adic valuation of $n$.
        \end{definition}
        \begin{definition}
        For a number $n$ we let $P(n)$ denote the largest prime factor of $n$.
        \end{definition}
		\section{\bf Lemmas and Backgrounds}

Let $m\in\mathbb{Z}$ and $T_m$ be the $m$-th term in the Tribonacci sequence. Then $f(x)=x^3-x^2-x-1$ denotes the characteristic polynomial of the Tribonacci numbers $T_m$. Note that $f$ has one real zero $\alpha$, and two complex conjugate zeros $\beta$ and $\gamma$.

Furthermore, we know
\begin{equation}\label{Binet}
	T_m=a\alpha^m+b\beta^m+c\gamma^m,
\end{equation}
where
\begin{equation*}
	\alpha=\frac{1+r_1+r_2}{3}, \qquad\beta=\frac{2-r_1-r_2+\sqrt{-3}\left(r_1-r_2\right)}{6}, \qquad\gamma=\overline{\beta},
\end{equation*}
\begin{equation*}
	a=\frac{1}{(\alpha-\beta)(\alpha-\gamma)}, \qquad b=\frac{1}{(\beta-\alpha)(\beta-\gamma)}, \qquad c=\overline{b},
\end{equation*}
where
\begin{equation*}
	r_1=\sqrt[3]{19+3\sqrt{33}}
\end{equation*}
and
\begin{equation*}
	r_2=\sqrt[3]{19-3\sqrt{33}}.
\end{equation*}
Furthermore, we have the following lower bound that can be proved by a straightforward induction valid for all $n\geq 1$:
\begin{equation*}
\alpha^{n-2}\leq T_n.
\end{equation*}

The following lemma is useful for the proofs of Theorem \ref{Thm-1} and Theorem \ref{Thm-3}.
\begin{lemma}\label{first-bound}
Let \begin{equation}\label{eq-tribo-gen}
	T_nT_{n+d}T_{n+2d}\cdots T_{n+rd}=m_1!m_2!\cdots m_k!,
\end{equation}
with $2\leq m_1\leq m_2\leq\ldots\leq m_k$, $d\geq 1$, and $n\geq 3$. Then
\begin{equation*}
	\left(n-2+\frac{rd}{2}\right)\left(r+1\right)<\frac{m_1\log m_1+\ldots+m_k\log m_k}{\log(1.83)}<\frac{3x\log(3x)}{\log(1.83)},
\end{equation*}
where $x:=\nu_2\left(m_1!\right)+\ldots+\nu_2\left(m_k!\right)$.
\end{lemma}

\begin{proof}[Proof of Lemma \ref{first-bound}]
Let \eqref{eq-tribo-gen} hold and $\alpha$ be the real zero of $f(x)$. Then we have $1.83^{n+di-2}<\alpha^{n+di-2}<T_{n+di}$ for all $0\leq i\leq r$.  Therefore, we have
\begin{equation*}
	1.83^{\left(n-2+\frac{rd}{2}\right)\left(r+1\right)}<T_nT_{n+d}T_{n+2d}\cdots T_{n+rd}=m_1!\cdots m_k!<m_1^{m_1}\cdots m_k^{m_k}.
\end{equation*}
Thus,
\begin{equation}\label{eqn1}
	\left(n-2+\frac{rd}{2}\right)\left(r+1\right)<\frac{m_1\log m_1+\ldots+m_k\log m_k}{\log(1.83)}.
\end{equation}
Let
\begin{equation}\label{eqn1.5}
x:=\nu_2\left(m_1!\right)+\ldots+\nu_2\left(m_k!\right)
\end{equation}
Notice that
\begin{equation*}
\nu_2(m!)=\left\lfloor\frac{m}{2}\right\rfloor+\left\lfloor\frac{m}{4}\right\rfloor+\ldots\geq\frac{m}{2}
\end{equation*}
holds for all $m\geq 2$, except for $m=3$, where we have $nu_2(3!)=1=3/3$. Thus, $m\leq 3\nu_2(m!)$ holds for all $m\geq 2$. Hence,
\begin{equation}\label{eqn2}
m_1\log m_1+\ldots+m_k\log m_k\leq(m_1+\ldots+m_k)\log(m_1+\ldots+m_k)\leq 3x\log(3x)
\end{equation}
Then, from \eqref{eqn1} and \eqref{eqn2}, we have
\begin{equation}\label{eqn29}
	\left(n-2+\frac{rd}{2}\right)\left(r+1\right)<\frac{m_1\log m_1+\ldots+m_k\log m_k}{\log(1.83)}<\frac{3x\log(3x)}{\log(1.83)}.
\end{equation}

\end{proof}
The following lemma that classifies the $2$-adic valuations of $T_n$ is due to Marques and Lengyel \cite{ML}.
\begin{lemma}\label{ML-Tribo-val}
	For $n\geq 1$, we have 
    \begin{displaymath}
    \nu_2(T_n)=
	\begin{cases}
		0, & n\equiv 1, 2 \pmod{4}\\
		1, & n\equiv 3, 11 \hspace{-0.2cm}\pmod{16}\\
		2, & n\equiv 4, 8 \pmod{16}\\
		3, & n\equiv 7 \quad \pmod{16}\\
		\nu_2(n)-1, & n\equiv 0 \quad \pmod{16}\\
		\nu_2(n+4)-1, & n\equiv 12 \ \pmod{16}\\
		\nu_2(n+17)+1, & n\equiv 15\ \pmod{32}\\
		\nu_2(n+1)+1, & n\equiv 31\ \pmod{32}\\
	\end{cases}
	\end{displaymath}
	
\end{lemma}

		
		\section{\bf Proof of the Theorems}
		In the following section, we prove the Theorems \ref{Thm-1}, \ref{Thm-2}, and \ref{Thm-3}.
		\subsection{Proof of Theorem \ref{Thm-1}}
		
		\begin{proof}
			We begin by assuming that
			\begin{equation*}
				T_nT_{n+1}T_{n+2}\cdots T_{n+r}=m_1!m_2!\cdots m_k!,
			\end{equation*}
			with $2\leq m_1\leq m_2\leq\ldots\leq m_k$ and $n\geq 3$.
			Then, from Lemma \ref{first-bound} with $d=1$, we have
			\begin{equation}\label{eqn3}
				\left(n-2+\frac{r}{2}\right)\left(r+1\right)<\frac{m_1\log m_1+\ldots+m_k\log m_k}{\log(1.83)}<\frac{3x\log(3x)}{\log(1.83)}.
			\end{equation}
			Also, we have
			\begin{equation*}
				\nu_2\left(T_nT_{n+1}T_{n+2}\cdots T_{n+r}\right)=\nu_2\left(T_n\right)+\nu_2\left(T_{n+1}\right)+\ldots+\nu_2\left(T_{n+r}\right).
			\end{equation*}
			Moreover,
			\begin{equation}\label{eqn4}
				\sum_{\substack{i=0\\n+i\equiv 1,2\Mod{4}}}^r\nu_2\left(T_{n+i}\right)=0,
			\end{equation}
			\begin{equation}\label{eqn5}
				\sum_{\substack{i=0\\n+i\equiv 3,11\Mod{16}}}^r\nu_2\left(T_{n+i}\right)\leq 2\left\lceil\frac{r+1}{16}\right\rceil,
			\end{equation}
			\begin{equation}\label{eqn16}
				\sum_{\substack{i=0\\n+i\equiv 7\Mod{16}}}^r\nu_2\left(T_{n+i}\right)\leq 3\left\lceil\frac{r+1}{16}\right\rceil,
			\end{equation}
			and
			\begin{equation}\label{eqn6}
				\sum_{\substack{i=0\\n+i\equiv 4,8\Mod{16}}}^r\nu_2\left(T_{n+i}\right)\leq 4\left\lceil\frac{r+1}{16}\right\rceil.
			\end{equation}
			Also,
			\begin{align}
				\sum_{\substack{i=0\\n+i\equiv 0\Mod{16}}}^r\nu_2\left(T_{n+i}\right)&=\sum_{\substack{i=0\\n+i\equiv 0\Mod{16}}}^r(\nu_2(n+i)-1)\nonumber\\
				&=\sum_{\substack{i=0\\n+i\equiv 0\Mod{16}}}^r\left(\nu_2\left(\frac{n+i}{16}\right)+3\right)\nonumber\\
				&\leq 3\left\lceil\frac{r+1}{16}\right\rceil+\sum_{i=5}^{\left\lfloor\frac{\log(n+r)}{\log 2}\right\rfloor}\left\lceil\frac{r+1}{2^i}\right\rceil\nonumber\\
				&\leq 3\left\lceil\frac{r+1}{16}\right\rceil+\frac{r+1}{16}+\frac{\log(n+r)}{\log 2}-4.\label{eqn7}
			\end{align}
			By the same reasoning,
			\begin{equation}\label{eqn8}
				\sum_{\substack{i=0\\n+i\equiv 12\Mod{16}}}^r\nu_2\left(T_{n+i}\right)\leq 3\left\lceil\frac{r+1}{16}\right\rceil+\frac{r+1}{16}+\frac{\log(n+r+4)}{\log 2}-4.
			\end{equation}
			By similar reasoning,
			\begin{equation}\label{eqn9}
				\sum_{\substack{i=0\\n+i\equiv 15\Mod{32}}}^r\nu_2\left(T_{n+i}\right)\leq 6\left\lceil\frac{r+1}{32}\right\rceil+\frac{r+1}{32}+\frac{\log(n+r+17)}{\log 2}-5
			\end{equation}
			and
			\begin{equation}\label{eqn10}
				\sum_{\substack{i=0\\n+i\equiv 31\Mod{32}}}^r\nu_2\left(T_{n+i}\right)\leq 6\left\lceil\frac{r+1}{32}\right\rceil+\frac{r+1}{32}+\frac{\log(n+r+1)}{\log 2}-5.
			\end{equation}
			Combining \eqref{eqn4} through \eqref{eqn10} we have
			\begin{align}
				x&=\nu_2\left(T_nT_{n+1}T_{n+2}\cdots T_{n+r}\right)\nonumber\\
				&\leq 15\left\lceil\frac{r+1}{16}\right\rceil+\frac{3(r+1)}{16}+12\left\lceil\frac{r+1}{32}\right\rceil+\frac{4\log(n+r+17)}{\log 2}-18\label{eqn10.5}\\
				&\leq\frac{3(r+1)}{2}+9+\frac{4\log(n+r+17)}{\log 2}\label{eqn11}.
			\end{align}
			Rearranging \eqref{eqn11} we have
			\begin{equation}\label{eqn12}
				2^{x/4-3(r+1)/8-9/4}-r-17\leq n.
			\end{equation}
			Also, \eqref{eqn3} gives
			\begin{equation}\label{eqn13}
				n<2-\frac{r}{2}+\frac{3x\log(3x)}{(r+1)\log(1.83)}.
			\end{equation}
			From \eqref{eqn3} we have
			\begin{equation*}
				\frac{r^2}{2}<\frac{3x\log(3x)}{\log(1.83)},
			\end{equation*}
			so that
			\begin{equation}\label{eqn14}
				r<\sqrt{\frac{6x\log(3x)}{\log(1.83)}}.
			\end{equation}
			Combining \eqref{eqn12}, \eqref{eqn13}, and \eqref{eqn14}, we have
			\begin{equation*}
				2^{x/4-9/4-\frac{3\left(\sqrt{\frac{6x\log(3x)}{\log(1.83)}}+1\right)}{8}}-17<\frac{1}{2}\sqrt{\frac{6x\log(3x)}{\log(1.83)}}+2+\frac{3x\log(3x)}{(r+1)\log(1.83)},
			\end{equation*}
			which implies $x\leq 246$. Therefore, \eqref{eqn14} implies $r\leq 127$. By \eqref{eqn13} we have
			\begin{equation}\label{eqn15}
				n<2-\frac{1}{2}+\frac{3\cdot 246\cdot\log(3\cdot 246)}{2\cdot\log(1.83)}<4034.
			\end{equation}
			\begin{case}{$n\geq 300$}
				\newline
				Here we eliminate the case that $n\geq 300$. By \eqref{eqn13} and the fact that $x\leq 246$, we have
				\begin{equation*}
					300<2-\frac{r}{2}+\frac{3\cdot 246\cdot\log(3\cdot 246)}{(r+1)\log(1.83)},
				\end{equation*}
				which gives $r\leq 24$. By \eqref{eqn10.5} and \eqref{eqn15}, we have
				\begin{equation*}
					x\leq 15\cdot\left\lceil\frac{25}{16}\right\rceil+\frac{3\cdot 25}{16}+12\left\lceil\frac{25}{32}\right\rceil+\frac{4\log(4033+17+24)}{\log 2}-18<77.
				\end{equation*}
				By \eqref{eqn13} and the fact that $x\leq 76$, we have
				\begin{equation*}
					300<2-\frac{r}{2}+\frac{3\cdot 76\cdot\log(3\cdot 76)}{(r+1)\log(1.83)},
				\end{equation*}
				which gives $r\leq 5$. By \eqref{eqn11} and \eqref{eqn15}, we have
				\begin{equation*}
					x\leq 15\cdot\left\lceil\frac{6}{16}\right\rceil+\frac{3\cdot 6}{16}+12\left\lceil\frac{6}{32}\right\rceil+\frac{4\log(4033+17+5)}{\log 2}-18<59.
				\end{equation*}
				By \eqref{eqn13} and the fact that $x\leq 58$, we have
				\begin{equation*}
					300<2-\frac{r}{2}+\frac{3\cdot 58\cdot\log(3\cdot 58)}{(r+1)\log(1.83)},
				\end{equation*}
				which gives $r\leq 3$. Since $n+r+17\leq 4033+3+17=4053$, we can see that $\nu_2\left(T_{n+i}\right)\leq 12$ for all $0\leq i\leq r$. Thus, we have $x\leq 12(r+1)\leq 48$. By \eqref{eqn13} and the fact that $x\leq 48$, we have
				\begin{equation*}
					300<2-\frac{r}{2}+\frac{3\cdot 48\cdot\log(3\cdot 48)}{(r+1)\log(1.83)},
				\end{equation*}
				which gives $r\leq 2$. Thus, we have $x\leq 12\cdot 3=36$. By \eqref{eqn13} and the fact that $x\leq 36$, we have
				\begin{equation*}
					300<2-\frac{r}{2}+\frac{3\cdot 36\cdot\log(3\cdot 36)}{(r+1)\log(1.83)},
				\end{equation*}
				giving $r=1$. Hence, $x\leq 12\cdot 2=24$. By \eqref{eqn13} and the fact that $x\leq 24$, we have
                \begin{equation*}
					300<2-\frac{r}{2}+\frac{3\cdot 24\cdot\log(3\cdot 24)}{(r+1)\log(1.83)},
				\end{equation*}
                giving $r<1$, eliminating this case.
			\end{case}
			Hence, we have $n\leq 299$. Since $x\leq 246$, we have $\nu_2\left(m_k!\right)\leq 246$. Thus, we have $m_k\leq 253$. Thus, we have $P\left(T_n\cdots T_{n+r}\right)\leq 251$. Checking via Maple, we see that the only Tribonacci numbers whose largest prime factor is at most 251 are $T_i$ for $2\leq i\leq 19$ and $i=28,31,34,35$. Then either $n+r\leq 19$ or $n=34$ and $r=1$. By Maple, we have
			\begin{equation*}
				T_{34}T_{35}=2\cdot 3^5\cdot 11\cdot 41\cdot 53^2\cdot 103^2\cdot 139\cdot 227,
			\end{equation*}
			eliminating that case. By Maple, we can see that $17$ does not divide $T_n$ for $2\leq n\leq 19$. Therefore, we have $m_k\leq 16$ and $P\left(T_nT_{n+1}\cdots T_{n+r}\right)\leq 13$. Again with the help of Maple we can see that $n+r\leq 9$. Therefore, we have $x\leq 8$. So $n+r\leq 5$, but then $3$ does not divide $T_n\cdots T_{n+r}$. So $T_n\cdots T_{n+r}$ is a power of $2$ and the only solution is $T_3T_4=2!\cdot 2!$.
		\end{proof}
		
\subsection{Proof of Theorem \ref{Thm-2}}
\begin{proof}
Let $n,r\in\mathbb{N}$ with $n\geq 5$ and assume that
\begin{equation*}
	T_{-n}T_{-n-1}T_{-n-2}\cdots T_{-n-r}=m_1!m_2!\cdots m_k!,
\end{equation*}
where $2\leq m_1\leq m_2\leq\ldots\leq m_k$.\\

We first note that $T_{-17}=0$. Therefore, if $n\leq 16$, we must have $n+r\leq 16$. The solutions in this case are those state in the theorem. Therefore, we may assume that $n\geq 18$.
\newline
		
For all $m\geq 18$, from \eqref{Binet} we have
		\begin{align}
			\left|T_{-m}\right|&=\left|a\alpha^{-m}+b\beta^{-m}+c\gamma^{-m}\right|\nonumber\\
			&\geq\left|b\beta^{-m}+c\gamma^{-m}\right|-a\alpha^{-m}\nonumber\\
			&\geq 2\left|\operatorname{Re}\left(b\beta^{-m}\right)\right|-a\alpha^{-18}\nonumber\\
			&>2\left|\operatorname{Re}\left(b\beta^{-m}\right)\right|-4\cdot 10^{-6}\nonumber\\
			&=2\left|b\beta^{-m}\right|\cdot\left|\cos\left(\arg\left(b\beta^{-m}\right)\right)\right|-4\cdot 10^{-6}\nonumber\\
			&=2|b||\beta|^{-m}\left|\cos\left(\arg b-m\arg\beta\right)\right|-4\cdot 10^{-6}\nonumber\\
			&>0.7\cdot 0.74^{-m}\left|\cos\left(\arg b-m\arg\beta\right)\right|-4\cdot 10^{-6}.\label{eqn17}
		\end{align}

        First suppose that $\left|T_{-m}\right|\leq 0.31\cdot 0.74^{-m}$. We will show that $\left|T_{-m-1}\right|>0.31\cdot 0.74^{-m-1}$ and $\left|T_{-m-2}\right|>0.31\cdot 0.74^{-m-2}$. If $\left|\cos\left(\arg b-m\arg\beta\right)\right|>0.46$, then \eqref{eqn17} gives 
		
		\begin{align}
			\left|T_{-m}\right|&>0.7\cdot 0.74^{-m}\cdot 0.46-4\cdot 10^{-6}\nonumber\\
			&>0.32\cdot 0.74^{-m}-4\cdot 10^{-6}\nonumber\\
			&>0.31\cdot 0.74^{-m}\label{eqn18}.
		\end{align}
		
		Hence, $\left|\cos\left(\arg b-m\arg\beta\right)\right|\leq 0.46$. Notice that for any real number $0<x<4$ such that $|\cos x|\leq 0.46$ we must have that $\left|x-\frac{\pi}{2}\right|<0.48$. From this fact it follows that there exists an integer $M$ such that

        \begin{equation}\label{eqn36}
        \left|\arg b-m\arg\beta-\frac{\pi}{2}+M\pi\right|<0.48.
        \end{equation}
		
		From the facts that $0.96<\pi-\arg\beta<0.97$ and $2.53<\frac{3\pi}{2}-\arg\beta<2.54$ and from \eqref{eqn36} we can therefore conclude that
		
		\begin{equation*}
			0.48<\arg b-(m+1)\arg\beta-\frac{\pi}{2}+(M+1)\pi<1.45,
		\end{equation*}
		
		and
		
		\begin{equation*}
			2.05<\arg b-(m+1)\arg\beta+(M+1)\pi<3.02.
		\end{equation*}

        Hence, 
        
        \begin{equation*}
            \left|\cos\left(\arg b-(m+1)\arg\beta\right)\right|=\left|\cos\left(\arg b-(m+1)\arg\beta\right)+(M+1)\pi\right|>0.46.
        \end{equation*}
        
        Thus, using the same argument in arriving at \eqref{eqn18}, we have
		
		\begin{equation}\label{eqn27}
			\left|T_{-m-1}\right|>0.31\cdot 0.74^{-m-1}
		\end{equation}

        Similarly, we can also see that 

        \begin{equation*}
			1.44<\arg b-(m+2)\arg\beta-\frac{\pi}{2}+(M+2)\pi<2.42,
		\end{equation*}
		
		and
		
		\begin{equation*}
			3.01<\arg b-(m+2)\arg\beta+(M+2)\pi<3.99.
		\end{equation*}

        Hence, 
        
        \begin{equation*}
            \left|\cos\left(\arg b-(m+2)\arg\beta\right)\right|=\left|\cos\left(\arg b-(m+2)\arg\beta\right)+(M+2)\pi\right|>0.46.
        \end{equation*}
		
		Thus,
		
		\begin{equation}\label{eqn28}
			\left|T_{-m-2}\right|>0.31\cdot 0.74^{-m-2}.
		\end{equation}
		
		Let $n\geq 18$ and $r\geq 1$ and consider $\left|T_{-n}T_{-n-1}\cdots T_{-n-r}\right|$. The above argument shows us that in the sequence $T_{-18},T_{-19},T_{-20},\ldots$ every three consecutive terms must contain at least two terms satisfying the inequality $\left|T_{-m}\right|>0.31\cdot 0.74^{-m}$. Also, for any terms not satisfying this inequality we have $\left|T_{-m}\right|\geq 1$. Thus,
		
		\begin{align}
			&\quad\left|T_{-n}T_{-n-1}\cdots T_{-n-r}\right|\nonumber\\
			&>0.31^{r+1-\lfloor\frac{r+1}{3}\rfloor}\cdot 0.74^{-n-(n+1)-(n+2)-\ldots-(n+r)+(n+2)+(n+5)+\ldots+\left(n-1+3\lfloor\frac{r+1}{3}\rfloor\right)+n}\nonumber\\
			&\geq 0.31^{r+1-\frac{r-1}{3}}\cdot 0.74^{-nr-\frac{r(r+1)}{2}+(n-1)\lfloor\frac{r+1}{3}\rfloor+\frac{3}{2}\lfloor\frac{r+1}{3}\rfloor\lfloor\frac{r+4}{3}\rfloor}\nonumber\\
			&\geq 0.31^{\frac{2(r+2)}{3}}\cdot 0.74^{-nr-\frac{r(r+1)}{2}+\frac{(n-1)(r+1)}{3}+\frac{(r+1)(r+4)}{6}}\nonumber\\
			&=0.31^{\frac{2(r+2)}{3}}\cdot 0.74^{-\frac{n(2r-1)+r^2-1}{3}}\label{eqn19}.
		\end{align}
		Letting $x$ be as in \eqref{eqn1.5} we have
		\begin{equation}\label{eqn20}
			2(r+2)\cdot\log(0.31)-\left(n(2r-1)+r^2-1\right)\log(0.74)<9x\log(3x)
		\end{equation}
		the analogous inequality to \eqref{eqn3}. Following the argument as before we obtain 
		\begin{align}
			x&=\nu_2\left|T_{-n}T_{-n-1}T_{-n-2}\cdots T_{-n-r}\right|\nonumber\\
			&=15\left\lceil\frac{r+1}{16}\right\rceil+\frac{3(r+1)}{16}+12\left\lceil\frac{r+1}{32}\right\rceil+\frac{4\log(n+r)}{\log 2}-18\label{eqn20.5}\\
			&\leq\frac{3(r+1)}{2}+9+\frac{4\log(n+r)}{\log 2}\label{eqn21}.
		\end{align}
		Thus,
		\begin{equation}\label{eqn22}
			2^{x/4-3(r+1)/8-9/8}-r\leq n.
		\end{equation}
		From \eqref{eqn20} we have
		\begin{equation}\label{eqn23}
			n<\frac{-9x\log(3x)}{(2r-1)\log(0.74)}+\frac{6\log(0.31)}{\log(0.74)}-\frac{r^2-1}{2r-1}.
		\end{equation}
		Combining \eqref{eqn22} and \eqref{eqn23} we have
		\begin{align}
			2^{x/4-3(r+1)/8-9/8}&<\frac{-9x\log(3x)}{(2r-1)\log(0.74)}+\frac{6\log(0.31)}{\log(0.74)}+\frac{r^2-r+1}{2r-1}\nonumber\\
			&=\frac{-9x\log(3x)}{(2r-1)\log(0.74)}+\frac{6\log(0.31)}{\log(0.74)}+\frac{r}{2}-\frac{1}{4}+\frac{3}{4(2r-1)}\label{eqn24}.
		\end{align}
		From \eqref{eqn20} we have 
		\begin{equation*}
			2(r+2)\cdot\log(0.31)-\left(18(2r-1)+r^2-1\right)\log(0.74)<9x\log(3x),
		\end{equation*}
		which gives
		\begin{equation*}
			\left(r+18-\frac{\log(0.31)}{\log(0.74)}\right)^2<-\frac{9x\log(3x)}{\log(0.74)}+\frac{4\log(0.31)}{\log(0.74)}+19+\left(18-\frac{\log(0.31)}{\log(0.74)}\right)^2,
		\end{equation*}
		from which we can deduce
		\begin{equation}\label{eqn25}
			r<\sqrt{30x\log(3x)+234}-14.
		\end{equation}
		Combining \eqref{eqn24} and \eqref{eqn25} gives
		\begin{equation*}
			2^{x/4-3(\sqrt{30x\log(3x)+234}-13)/8-9/8}<\frac{-9x\log(3x)}{\log(0.74)}+\frac{6\log(0.31)}{\log(0.74)}+\frac{\sqrt{30x\log(3x)+234}-14}{2}+\frac{1}{2},
		\end{equation*}
		which gives $x\leq 609$, so that $r\leq 356$. From \eqref{eqn23} we have
		\begin{equation}\label{eqn26}
			n<\frac{-9\cdot 609\log(3\cdot 609)}{(2\cdot 1-1)\log(0.74)}+\frac{6\log(0.31)}{\log(0.74)},
		\end{equation}
		which gives $n\leq 136735$.
		\begin{case}{$n\geq 500$}
			\newline
			Here we eliminate the case $n\geq 500$. By \eqref{eqn23} and the fact that $x\leq 609$, we have
			\begin{equation*}
				500<\frac{-9\cdot 609\log(3\cdot 609)}{(2r-1)\log(0.74)}+\frac{6\log(0.31)}{\log(0.74)}-\frac{r^2-1}{2r-1},
			\end{equation*}
			which gives $r\leq 126$. By \eqref{eqn20.5} and \eqref{eqn25}, we have
			\begin{equation*}
				x\leq 15\cdot\left\lceil\frac{127}{16}\right\rceil+\frac{3\cdot 127}{16}+12\left\lceil\frac{127}{32}\right\rceil+\frac{4\log(136735+126)}{\log 2}-18<216.
			\end{equation*}
			By \eqref{eqn23} and the fact that $x\leq 215$, we have
			\begin{equation*}
				n<\frac{-9\cdot 215\log(3\cdot 215)}{(2\cdot 1-1)\log(0.74)}+\frac{6\log(0.31)}{\log(0.74)}<41597
			\end{equation*}
			and
			\begin{equation*}
				500<\frac{-9\cdot 215\log(3\cdot 215)}{(2r-1)\log(0.74)}+\frac{6\log(0.31)}{\log(0.74)}-\frac{r^2-1}{2r-1},
			\end{equation*}
			which gives $r\leq 42$. By \eqref{eqn20.5} and \eqref{eqn25}, we have
			\begin{equation*}
				x\leq 15\cdot\left\lceil\frac{43}{16}\right\rceil+\frac{3\cdot 43}{16}+12\left\lceil\frac{43}{32}\right\rceil+\frac{4\log(41596+42)}{\log 2}-18<121.
			\end{equation*}
			By \eqref{eqn23} and the fact that $x\leq 120$, we have
			\begin{equation*}
				n<\frac{-9\cdot 120\log(3\cdot 120)}{(2\cdot 1-1)\log(0.74)}+\frac{6\log(0.31)}{\log(0.74)}<21136
			\end{equation*}
			and
			\begin{equation*}
				500<\frac{-9\cdot 120\log(3\cdot 120)}{(2r-1)\log(0.74)}+\frac{6\log(0.31)}{\log(0.74)}-\frac{r^2-1}{2r-1},
			\end{equation*}
			which gives $r\leq 22$. By \eqref{eqn20.5} and \eqref{eqn25}, we have
			\begin{equation*}
				x\leq 15\cdot\left\lceil\frac{23}{16}\right\rceil+\frac{3\cdot 23}{16}+12\left\lceil\frac{23}{32}\right\rceil+\frac{4\log(21135+22)}{\log 2}-18<86.
			\end{equation*}
			By \eqref{eqn23} and the fact that $x\leq 85$, we have
			\begin{equation*}
				n<\frac{-9\cdot 85\log(3\cdot 85)}{(2\cdot 1-1)\log(0.74)}+\frac{6\log(0.31)}{\log(0.74)}<14102
			\end{equation*}
			and
			\begin{equation*}
				500<\frac{-9\cdot 85\log(3\cdot 85)}{(2r-1)\log(0.74)}+\frac{6\log(0.31)}{\log(0.74)}-\frac{r^2-1}{2r-1},
			\end{equation*}
			which gives $r\leq 15$. By \eqref{eqn20.5} and \eqref{eqn25}, we have
			\begin{equation*}
				x\leq 15\cdot\left\lceil\frac{15}{16}\right\rceil+\frac{3\cdot 15}{16}+12\left\lceil\frac{15}{32}\right\rceil+\frac{4\log(14101+15)}{\log 2}-18<67.
			\end{equation*}
			By \eqref{eqn23} and the fact that $x\leq 66$, we have
			\begin{equation*}
				n<\frac{-9\cdot 66\log(3\cdot 66)}{(2\cdot 1-1)\log(0.74)}+\frac{6\log(0.31)}{\log(0.74)}<10456
			\end{equation*}
			and
			\begin{equation*}
				500<\frac{-9\cdot 66\log(3\cdot 66)}{(2r-1)\log(0.74)}+\frac{6\log(0.31)}{\log(0.74)}-\frac{r^2-1}{2r-1},
			\end{equation*}
			which gives $r\leq 11$. At this point, instead of using \eqref{eqn20.5} and \eqref{eqn25} again, we may simply use Lemma \ref{ML-Tribo-val} with the observation that $n+r\leq 10456+11<2^{14}$ to conclude that $x\leq 14+2\cdot 12+3+2+1=44$. Then, by \eqref{eqn23}, we have
			\begin{equation*}
				n<\frac{-9\cdot 44\log(3\cdot 44)}{(2\cdot 1-1)\log(0.74)}+\frac{6\log(0.31)}{\log(0.74)}<6445
			\end{equation*}
			and
			\begin{equation*}
				500<\frac{-9\cdot 44\log(3\cdot 44)}{(2r-1)\log(0.74)}+\frac{6\log(0.31)}{\log(0.74)}-\frac{r^2-1}{2r-1},
			\end{equation*}
			which gives $r\leq 7$. Since $n+r<2^{13}$, we have $x\leq 13+2\cdot 11+1=36$. Then, by \eqref{eqn23}, we have
			\begin{equation*}
				n<\frac{-9\cdot 36\log(3\cdot 36)}{(2\cdot 1-1)\log(0.74)}+\frac{6\log(0.31)}{\log(0.74)}<5062
			\end{equation*}
			and
			\begin{equation*}
				500<\frac{-9\cdot 36\log(3\cdot 36)}{(2r-1)\log(0.74)}+\frac{6\log(0.31)}{\log(0.74)}-\frac{r^2-1}{2r-1},
			\end{equation*}
			which gives $r\leq 5$. So $x\leq 13+2\cdot 11=35$. From $x\leq 35$, we deduce $m_k\leq 39$. From \eqref{eqn19}, we have
			\begin{equation*}
				\left|T_{-n}T_{-n-1}\cdots T_{-n-r}\right|>0.31^{r+1-\frac{r-1}{3}}\cdot 0.74^{-500r-\frac{r(r+1)}{2}+499\lfloor\frac{r+1}{3}\rfloor+\frac{3}{2}\lfloor\frac{r+1}{3}\rfloor\lfloor\frac{r+4}{3}\rfloor}
			\end{equation*}
			Evaluating the above lower bound for $1\leq r\leq 5$, we find
			\begin{equation*}
				\left|T_{-n}T_{-n-1}\cdots T_{-n-r}\right|>3\cdot 10^{64}.
			\end{equation*}
			Notice that
			\begin{equation*}
				\frac{\log m_1!+\ldots+\log m_k!}{x\log 2}\geq \frac{\log\left|T_{-n}T_{-n-1}\cdots T_{-n-r}\right|}{35\log 2}>\frac{\log\left(3\cdot 10^{64}\right)}{35\log 2}>6.11.
			\end{equation*}
			Hence, there exists some $1\leq m\leq 39$ such that
			\begin{equation*}
				\frac{\log m!}{\nu_2(m!)\log 2}>6.11.
			\end{equation*}
			Maple shows that no such $m$ exists, eliminating the case $n\geq 500$.
		\end{case}
		Hence, we have $n\leq 499$. Since $x\leq 609$, we have $\nu_2\left(m_k!\right)\leq 609$. Thus, we have $m_k\leq 615$. Therefore, $P\left(\left|T_{-n}T_{-n-1}\cdots T_{-n-r}\right|\right)\leq 613$. Checking via Maple, we see that the only Tribonacci numbers $T_{-m}$, where $18\leq m\leq 499$ whose largest prime factor is at most $613$ are $T_{-m}$ for $18\leq m\leq 28$, $m=30$, $32\leq m\leq 36$, $m=38$, $40\leq m\leq 41$, $m=43$, $m=46$, $49\leq m\leq 52$, $55\leq m\leq 57$, $m=63$, $65\leq m\leq 66$, and $68\leq m\leq 69$. 
		\newline
		\newline
		For all $18\leq m\leq 28$, $49\leq m\leq 52$, and $55\leq m\leq 57$, we see that $5\nmid T_{-m}$, but there exists a prime factor $p\geq 7$ such that $p\mid T_{-m}$.
		\newline
		\newline
		For all $32\leq m\leq 36$, we see that $11\nmid T_{-m}$. For all $34\leq m\leq 36$ there exists a prime factor $p\geq 13$ such that $p\mid T_{-m}$. Also, we have $5^2\mid\left|T_{-32}T_{-33}\right|$, but $3^2\nmid\left|T_{-32}T_{-33}\right|$.
		\newline
		\newline
		Also, we have $7\nmid\left|T_{-40}T_{-41}\right|$, but $23\mid \left|T_{-40}T_{-41}\right|$.
		\newline
		\newline
		As well, we have $11\nmid\left|T_{-65}T_{-66}\right|$, but $19\mid \left|T_{-65}T_{-66}\right|$.
		\newline
		\newline
		Finally, we have $5\nmid\left|T_{-68}T_{-69}\right|$, but $7\mid \left|T_{-68}T_{-69}\right|$.
		\newline
		\newline
		Hence, there are no solutions to \ref{Tribo-2} with $r\geq 1$ and $n\geq 18$, and this finishes the proof.
\end{proof}
\subsection{Proof of Theorem \ref{Thm-3}}
		\begin{proof}
Assume that
\begin{equation*}
	T_nT_{n+d}T_{n+2d}\cdots T_{n+rd}=m_1!m_2!\cdots m_k!,
\end{equation*}
with $2\leq m_1\leq m_2\leq\ldots\leq m_k$, $d\geq 2$, and $n\geq 3$. Then, from Lemma \ref{first-bound},
\begin{equation*}
	\left(n-2+\frac{rd}{2}\right)\left(r+1\right)<\frac{m_1\log m_1+\ldots+m_k\log m_k}{\log(1.83)}<\frac{3x\log(3x)}{\log(1.83)}.
\end{equation*}			
			Also, we have
			\begin{equation*}
				\nu_2\left(T_nT_{n+d}T_{n+2d}\cdots T_{n+rd}\right)=\nu_2\left(T_n\right)+\nu_2\left(T_{n+d}\right)+\ldots+\nu_2\left(T_{n+rd}\right).
			\end{equation*}
			By Lemma \ref{ML-Tribo-val}, we can see that for all $0\leq i\leq r$, we have
            \begin{equation*}
            \nu_2\left(T_{n+di}\right)\leq\frac{\log(n+dr+17)}{\log 2}+1.
            \end{equation*}
            Thus,
			\begin{equation}\label{eqn30}
				x=\sum_{i=0}^r\nu_2\left(T_{n+di}\right)\leq(r+1)\left(\frac{\log(n+rd+17)}{\log 2}+1\right).
			\end{equation}
			Rearranging \eqref{eqn30} we have
			\begin{equation}\label{eqn31}
				2^{x/(r+1)-1}-rd-17\leq n.
			\end{equation}
			Also, \eqref{eqn29} gives
			\begin{equation}\label{eqn32}
				n<2-\frac{rd}{2}+\frac{3x\log(3x)}{(r+1)\log(1.83)},
			\end{equation}
			so that 
			\begin{equation}\label{eqn33}
				rd<\frac{3x\log(3x)}{\log(1.83)}.
			\end{equation}
			From \eqref{eqn29} we have
			\begin{equation*}
				\frac{r^2d}{2}<\frac{3x\log(3x)}{\log(1.83)},
			\end{equation*}
			so that
			\begin{equation}\label{eqn34}
				r<\sqrt{\frac{6x\log(3x)}{d\log(1.83)}}.
			\end{equation}
			Combining \eqref{eqn31}-\eqref{eqn34}, we have
			\begin{equation*}
				2^{\frac{x}{\left(\sqrt{\frac{6x\log(3x)}{d\log(1.83)}}+1\right)}-1}<19+\frac{x\log(3x)}{2\log(1.83)}+\frac{3x\log(3x)}{(r+1)\log(1.83)},
			\end{equation*}
			which implies $x\leq 25769$. Therefore, \eqref{eqn34} implies $r\leq 1199$. As well, \eqref{eqn33} implies $rd\leq 1439862$. By \eqref{eqn32}, we have
			\begin{equation}\label{final}
				n<2-\frac{1}{2}+\frac{3\cdot 25769\cdot\log(3\cdot 25769)}{2\cdot\log(1.83)}<719933.
			\end{equation}
			
		\end{proof}

\section{Acknowledgments}
This paper was supported by Title III, New Faculty Funding by Xavier University of Louisiana and by MT Internal Grant Opportunities from Middle Tennessee State University.

\printbibliography
\end{document}